\documentclass[pdflatex,sn-mathphys-num]{sn-jnl}
\usepackage{amsmath, amsthm, amssymb, epsfig}
\usepackage{subfig}
\usepackage{mathtools}
\usepackage{mathrsfs}
\usepackage[dvipsnames]{xcolor}
\usepackage{color}
\usepackage{url}
\usepackage{algorithmicx}
\usepackage[ruled]{algorithm}
\usepackage{algpseudocode}
\usepackage{graphicx}
\usepackage{bbold}
\usepackage{fullpage} 
\usepackage{hyperref}
\usepackage{wrapfig}
\usepackage{mathtools}
\usepackage{placeins}

\newtheorem{theorem}{Theorem}
\newtheorem{proposition}{Proposition} 
\newtheorem{lemma}{Lemma}

\newtheorem{corollary}[theorem]{Corollary}

\theoremstyle{definition}
\newtheorem{assumption}{Assumption}
\newtheorem{definition}{Definition}

\newcommand{\edit}[1]{{#1}}

\newcommand{\m}[1]{\mathbb{#1}}

\def \Re {\mathbb{R}}

\def \E { \mathbb{E}}

\def \proj {\mathcal{P}_{\mathcal{W}} }

\newcommand{\f}[2]{\frac {#1}{#2}}
\newcommand{\norm}[1]{\left\lVert#1\right\rVert}

\newcommand{\fold}[1]{\text{fold}\left(#1\right)}
\newcommand{\unfold}[1]{\text{unfold}\left(#1\right)}

\newcommand{\bcirc}[1]{\text{bcirc}\left(#1\right)}

\newcommand{\bdiag}[1]{\text{bdiag}\left(#1\right)}

\newcommand{\cA}{{\cal A}}
\newcommand{\ctA}{{\cal {\tilde A}}}
\newcommand{\cB}{{\cal B}}
\newcommand{\cC}{{\cal C}}
\newcommand{\cD}{{\cal D}}
\newcommand{\cE}{{\cal E}}

\newcommand{\cH}{{\cal H}}
\newcommand{\cI}{{\cal I}}

\newcommand{\cP}{{\cal P}}
\newcommand{\cQ}{{\cal Q}}
\newcommand{\cR}{{\cal R}}

\newcommand{\cT}{{\cal T}}

\newcommand{\cX}{{\cal X}}
\newcommand{\cY}{{\cal Y}}
\newcommand{\cW}{{\cal W}}

\newcommand{\mA}{{\bf A}}

\newcommand{\mH}{{\bf H}}

\newcommand{\mX}{{\bf X}}
\newcommand{\mY}{{\bf Y}}

\newcommand\numberthis{\addtocounter{equation}{1}\tag{\theequation}}

\title{Stochastic Gradient Descent for Incomplete Tensor Linear Systems}

\author[1]{\fnm{Anna} \sur{Ma}} \email{anna.ma@uci.edu}

\author[2]{\fnm{Deanna} \sur{Needell}}\email{deanna@math.ucla.edu}

\author*[2]{\fnm{Alexander} \sur{Xue}}\email{alexxue@math.ucla.edu
}

\affil[1]{University of California, Irvine, Department of Mathematics}
\affil[2]{University of California, Los Angeles, Department of Mathematics}

\abstract{
    Solving large tensor linear systems poses significant challenges due to the high volume of data stored, and it only becomes more challenging when some of the data is missing. Recently, Ma et al. showed in~\cite{Ma2025Stocha} that this problem can be tackled using a stochastic gradient descent-based method, assuming that the missing data follows a uniform missing pattern. We adapt the technique by modifying the update direction, showing that the method is applicable under other missing data models. We prove convergence results and experimentally verify these results on synthetic data.
}
\keywords{tensor recovery, t-product, missing data, stochastic gradient descent}
\pacs[MSC Classification]{65F10, 15A69, 65K10}
\begin{document}
\maketitle
\section{Introduction} 
In recent years, the rapid growth of data-driven applications in fields such as computer vision, signal processing, and machine learning has led to a growing demand for models and algorithms capable of efficiently handling high-dimensional data. Traditional matrix-based approaches often fall short in capturing the intrinsic multi-way relationships present in such datasets. As a result, \textit{tensors}, which are multidimensional generalizations of matrices, have emerged as a natural framework for representing and manipulating this high-dimensional data.

Tensor methods have shown significant promise in various machine learning applications, including deep learning~\cite{Novikov2015} and recommender systems~\cite{Karatzoglou2010}. In these settings, some high-dimensional data can be represented by a tensor with no expressivity power loss. At the same time, tensors help with efficiency and compactness via various methods like tensor decomposition and compression. 

Recently, iterative methods involving tensors have been studied. These aim to solve tensor linear systems under the \textit{t-product}~\cite{KILMER2011641}.
For instance,~\cite{tang2023sketch} adapts the sketch-and-project framework to tensors,~\cite{luo2024frontal} introduces a novel frontal slice descent method, and~\cite{castillo2024randomized} modifies the Randomized Kaczmarz method to the tensor case to solve a tensor system with factorized operators. As can be seen from these examples, the $t$-product enables concepts from linear algebra in the matrix case to be adapted to the tensor case. Accordingly, the t-product has seen a wide array of applications, including in computer vision~\cite{yin2019subspaceclustering}, neural networks~\cite{newman2018stabletensorneuralnetworks}, and data completion~\cite{hu2017videocompletion}, among many others.

In this paper, we consider the problem of solving linear systems for third-order tensors, assuming that some of the data in a tensor is missing. This is a common problem in applications where, e.g., the data stems from inactive sensors, incomplete survey data, collaborative filtering, or memory needs. 

Our main contributions can be summarized as follows.

\begin{enumerate}
    \item We develop a new stochastic gradient descent-based algorithm to solve incomplete tensor linear systems under the t-product.

    \item Our results generalize the results in~\cite{ma2019stochasticgradientdescentlinear} from the matrix case to the tensor case. In doing so, we highlight the nontrivial distinctions one must make in the tensor setting, both algorithmically, analytically, and in terms of the noise models. 
    The recent paper~\cite{Ma2025Stocha} also tackles the tensor case, but we present a general approach that is applicable to a variety of missing data models, whereas their method is specifically tailored to a particular missing data model.
    
    \item We introduce three missing data models for which our proposed approach applies, providing rigorous error estimates for these models under both constant and varying step sizes.
\end{enumerate}

\textbf{Organization.} The rest of this section is dedicated to introducing tensors and the problem that we are studying. Section~\ref{sec:sgd} discusses the algorithm that we develop for the problem and gives two convergence results, Section~\ref{sec:convergence} proves the results and showcases the results for three specific missing data models, and Section~\ref{sec:experiments} empirically verifies the theoretical results on synthetic data.

\subsection{Background on Tensors}

\textbf{Notation.} Tensors are denoted with calligraphic letters, matrices are denoted with boldface capital letters, and vectors are denoted with lowercase letters. We focus on real third-order tensors in this paper. The $(i, j, k)$th element of a tensor $\cA \in \m R^{m \times \ell \times n}$ is denoted by $\cA_{ijk}.$ We use zero-indexing in this paper -- thus $0 \le i < m, 0 \le j < \ell,$ and $0 \le k < n$ in the previous sentence. \textit{Slices} of tensors are sub-tensors obtained by holding one index fixed, e.g. $\cA_{i::}$, $\cA_{:j:},$ and $\cA_{::k}$ denote the $i$th row slice, $j$th column slice, and $k$th frontal slice of $\cA$, respectively.

Since it will be convenient to sometimes view the frontal slices as matrices, we use the bold capital letter notation $\mA_k$ to do so.
Frequently, \edit{we will be examining sub-tensors of a row slice $\cA_{i::} \in \mathbb{R}^{1 \times \ell \times n}$. Note that the $k^{th}$ frontal slice of $\cA_{i::}$, i.e., $\cA_{i:k} \in \mathbb{R}^{1 \times \ell \times 1}$ is a vector. To simplify the notation, assuming a fixed $i \in [m]$, we denote the $k$th frontal slice of $\cA_{i::} $ as a vector via the lower case notation $a_k.$ That is, $a_k$ is the vector in $\m R^{\ell}$ whose $j$th entry is $\cA_{ijk}$, for $0 \le j < \ell.$}

\begin{definition} [Tensor operations] 
Let $\cA \in \m R^{m \times \ell \times n}$.
Let $\bcirc{\cA} \edit{\in \m R^{mn \times \ell n}}$ be the block-circulant matrix created from the frontal slices of $\cA$. That is,
\[ \bcirc{\cA} := \begin{pmatrix} \mA_0 & \mA_{n-1} & \mA_{n-2} & \dots & \mA_1 \\ 
\mA_1 & \mA_0 & \mA_{n-1} & \dots & \mA_2 \\ 
\vdots & \vdots & \vdots & \ddots & \vdots \\
\mA_{n-1} & \mA_{n-2} & \mA_{n-3} & \dots & \mA_0
\end{pmatrix}.\]
Let $\unfold{\cA}$ be the flattened matrix version of $\cA$, i.e. it is the $mn \times \ell$ matrix
\[ \unfold{\cA} := \begin{pmatrix}
    \mA_0 \\ \mA_1 \\ \vdots \\ \mA_{n-1}
\end{pmatrix}.\]
Also let $\fold{\cdot}$ be the function that reverses $\unfold{\cdot}$, so that $\fold{\unfold{\cA}} = \cA.$ Note that fold requires the number of frontal slices in order to do the unfolding. All of our tensors will have the same number $n$ of frontal slices, so there is no ambiguity in how fold works.
\end{definition}

\begin{definition} [t-product]
Let $\cA \in \m R^{m \times \ell \times n}$.
We can now define the t-product. Let $\cX$ be a $\m R^{\ell \times q \times n}$ tensor. The t-product, as defined in~\cite{KILMER2011641}, between $\cA$ and $\cX$ is \edit{the tensor in $\m R^{m \times q \times n}$ given by}
\[ \cA \ast \cX := \fold { \bcirc{\cA} \unfold{\cX}},\]
where the product $\bcirc{\cA}\unfold{\cX}$ is the usual matrix-matrix product.
\end{definition}

\begin{definition} [Transpose, Hermitian]
The transpose $\cA^{\edit{T}}$ of a tensor is the tensor \edit{in $\m R^{\ell \times m \times n}$} obtained by taking the transpose of every frontal slice \edit{of $\cA$} and also reversing the order of the frontal slices $1, 2, \dots, n-1.$ In other words,
\[ \unfold{\cA^{\edit{T}}} = \begin{pmatrix}
    \mA_0^{\edit{T}} \\ \mA_{n-1}^{\edit{T}} \\ \vdots \\ \mA_1^{\edit{T}}
\end{pmatrix}.\]
When using both slice and transpose notation, we apply the slice first then apply the transpose, e.g. $\cA_{i::}^{\edit{T}} := (\cA_{i::})^{\edit{T}}.$

$\cA$ is said to be Hermitian if $\cA = \cA^{\edit{T}}.$
\end{definition}

\begin{definition} [Norm]
The inner product of two tensors (of the same dimension) is the sum of the element-wise products of the tensors,
$$ \langle \cA, \cB \rangle := \sum_{i, j, k} \cA_{ijk}\cB_{ijk}.$$
Then, we define the norm $\|\cA\|$ to be the Frobenius norm,
$$ \|\cA\| := \sqrt{\langle \cA, \cA \rangle} = \sqrt{\sum_{i,j,k} \cA_{ijk}^2}.$$
Since we view slices of tensors as sub-tensors, the norm of slices is also the Frobenius norm, \edit{e.g.}
$$ \|\cA_{i::}\| = \sqrt {\sum_{j, k} \cA_{ijk}^2}.$$

\end{definition}

\subsection{Tensor Linear Systems with Missing Data}

Let $\cA \in \m R^{m \times \ell \times n}, \cX \in \m R^{\ell \times q \times n},$ and $\cB \in \m R^{m \times q \times n}$ be third-order tensors. We are interested in solving the tensor linear system $\cA \ast \cX = \cB$ given $\cB$ and only partial knowledge of $\cA.$ Here, $\cA$ represents the data, $\cX$ represents the unknown variables, and $\cB$ represents the measurements. 

Our method is a general method that applies to a variety of missing data models. To showcase this, we consider three specific missing data models, including the uniform model studied in~\cite{Ma2025Stocha}. Let $\tilde \cA$ denote the tensor of observed elements, and let $\cD \in \{0, 1\}^{m \times \ell \times n}$ be a \textit{binary mask} indicating which elements we observe, so that $\tilde \cA = \cD \circ \cA$, where $\circ$ indicates the element-wise product. The three missing data models are as follows.
\begin{enumerate}
    \item (Uniform missing model) Every entry of $\cA$ is missing independently with a probability $1 - p$. Formally, the entries of $\cD$ are i.i.d. Bernoulli random variables with parameter $p$. 
    Note that this uniform missing model was the model used in~\cite{Ma2025Stocha}. By generalizing the methods in that paper, we obtain their results as a corollary.
    \item (Column block missing model) Let $b$ be a divisor of the number of column slices $\ell$. Then, for each row slice of $\cA$, each column block of size $b$ is missing with a probability $1 - p.$ Formally, the column blocks $\cD_{i:, 0:b-1, :}, \cD_{i:, b:2b-1, :}, \dots$ of each row slice of $\cD$ are i.i.d variables equaling the tensor of all ones with probability $p$ and the zero tensor with probability $1 - p$.
    \item (Frontal slice missing model) For each row slice of $\cA$, each frontal slice is missing independently with a probability $1 - p$. Formally, the frontal slices of each row slice of $\cD$ are i.i.d variables equaling the tensor of all ones with probability $p$ and the zero tensor with probability $1 - p.$
\end{enumerate}

We remark that the the first two models have natural analogues in the matrix case, but the third has no such analogue. Indeed, uniformly missing entries and missing column blocks make sense in a matrix setting, but missing frontal slices is unique in the tensor setting. Despite this, our technique still applies to this unique problem just as easily as it applies to the first two.

 Under any missing data model, the problem that we study can be written as an optimization program
\begin{align*}
    &\text{Given } \tilde \cA = \cD \circ \cA, \cB, \text{ and } p \\
    &\text{Find } \cX_\star = \mathrm{argmin}_{\cX \in \cW}  \f 1 {2m} \|\cA \ast \cX - \cB\|^2,
\end{align*}
where $\cW$ is a convex domain containing the solution $\cX_\star$ of $\cA \ast \cX = \cB.$
\edit{Here, the $1/2$ scaling factor is for convenience to cancel with the squared term when taking the gradient, and the $1/m$ scaling factor averages the objective over the number of row slices. It should be noted that this scaling does not change the overarching minimization problem.}

\section{Stochastic Gradient Descent for Tensors with Missing Data} \label{sec:sgd}
SGD is a popular iterative method used to minimize a convex objective function $F(\cX)$ over a convex domain $\cW$. In particular, SGD has been useful in machine learning applications, such as in~\cite{amiri2020machine} and~\cite{bottou2010large}.

SGD works by using an estimator of the gradient of an objective function to determine what direction to step in at each iteration. If $g$ is a random function providing this estimate, then SGD proposes to update iterates via
\begin{equation} \label{eq:SGDupdate}
\cX^{t+1} = \cP_\cW(\cX^t - \alpha_t g(\cX^t)),
\end{equation}
where $\alpha_t$ is an appropriately chosen learning rate or step size, and $\cP_\cW$ is the projection onto a convex set $\cW$. The choice of $\alpha_t$ in SGD is important. A smaller $\alpha_t$ makes it more likely that the step is a descent step (i.e., one that lowers the value of the objective function) but slows down convergence. Thus, there is a tradeoff between the speed of convergence and the accuracy of convergence.

Previous works have used variants of SGD for tensor decomposition~\cite{maehara2016expected, kolda2020stochastic}, completion~\cite{papastergiou2017tensorcompletion}, and recovery~\cite{chen2021tensorrecovery,ma2020randomizedkaczmarztensorlinear}. For example,~\cite{grotheer2023iterativesingulartubehard} combines SGD with thresholding for sparse tensor recovery and~\cite{papastergiou2017tensorcompletion} introduces proximal SGD-based algorithms for tensor completion. Randomized Kaczmarz, a special case of SGD, has also been studied on tensors. \cite{chen2021tensorrecovery} applies a tensor version of Kaczmarz using a learning rate proportional to $1/\|\cA\|^2$ to solve tensor recovery problems. In contrast, \cite{ma2020randomizedkaczmarztensorlinear} proposes a more complex update step by incorporating a tensor inverse. 
It should be noted that these works either solve tensor recovery with all data or deal with missing data (e.g., via tensor completion) independently. In this work, we tackle the recovery and missing data problems simultaneously. 

With SGD, we can obtain various convergence results depending on additional properties of $F$ or $g$. Convergence results for SGD typically require that $g$ provides an unbiased estimate for the gradient of $F$, that is, $\m E[g(\cX)] = \nabla F(X).$ For example, the following lemma, which we use for one of our convergence results, has it as a requirement.

\begin{lemma}(\cite{shamir2013stochastic} Theorem 2) Suppose that $F$ is convex and $\cW$ is a closed convex domain containing the solution $\cX_\star$. Furthermore, suppose that for some constants $G$ and $K$, it holds that $\m E[g(\cX)] = \nabla F(\cX)$ and $\m E[\|g(\cX)\|^2] \le G$ for all $\cX \in \cW$, and $\sup_{\cW_1, \cW_2 \in \cal W} \| \cW_1 - \cW_2\| \le K$. Using step size $\alpha_t = C / \sqrt t$ where $C>0$ is a constant, and using the SGD update in equation~\ref{eq:SGDupdate}, the resulting iterates satisfy
$$\E [ F(\cX^{t}) - F(\cX_\star) ] \leq \left ( \f {K^2} C + C G \right) \f {2 + \log (t) } {\sqrt t}.$$
\label{lem:sgdbound}
\end{lemma}
 Here, $C$ represents the step size constant factor, $G$ gives a bound on the expected squared norm of $g$, and $K$ gives a bound on the distance between any two elements $\cW_1$ and $\cW_2$ of $\cW$, which can be interpreted as a bound on the size of $\cW.$

Let
\[F(\cX) := \f 1 {2m} \| \cA \ast \cX - \cB\|^2\] be our objective function. The main challenge in applying this lemma to our missing data setting is in choosing a suitable $g$ that provides an unbiased estimate for $\nabla F$, as we cannot directly compute $\nabla F$ without knowing $\cA.$

For our analysis, we require the following assumptions. Assumption~\ref{assump:row} is a technical one unrelated to the model. It is needed to make the randomness in $g$ easier to deal with. Assumption~\ref{assump:model} is an assumption on the model that simplifies the computation of some expectations.

\begin{assumption} \label{assump:row} Either every row $\cA_{i::}$ is chosen at most once during Algorithm~\ref{alg:mSGDT}, or the mask is redrawn at each iteration, meaning the data that is available from each row may change with time. If $\cA$ has a large number of rows, it is reasonable to assume that the former is satisfied.
\end{assumption} 
\begin{assumption} \label{assump:model} Every entry of $\cA$ occurs with a probability $p$. That is,
$\m E [ \tilde \cA] = p\, \m E[\cA].$ Note that the missingness of the entries do not have to be independent of each other.
\end{assumption}
Under these assumptions, we apply SGD using an update function $g$ in the form
\begin{align} 
g(\cX) := &\frac{1}{p^2} \left(\ctA_{i::}^{\edit{T}} \ast (\ctA_{i::}\ast\cX - p \cB_{i::}) \right) - \frac{1-p}{p^2} \cC \circ (\ctA_{i::}^{\edit{T}}\ast\ctA_{i::})\ast\cX, \label{eqn:g}
\end{align}
where $\circ$ is the element-wise product and $\cC$ is a tensor that depends on the assumed missing data model. Note that the element-wise product $\cC \circ (\ctA_{i::}^{\edit{T}}\ast\ctA_{i::})$ is computed before the t-product with $\cX$.
\edit{Alternatively, if we had chose to define $\ctA$ by scaling each nonzero entry by $1/p$ so that $\m E[\ctA] = \m E[\cA]$, the update function would become 
\begin{align*} 
g_{other}(\cX) :=  \left(\ctA_{i::}^{\edit{T}} \ast (\ctA_{i::}\ast\cX - p \cB_{i::}) \right) - (1-p) \cC \circ (\ctA_{i::}^{\edit{T}}\ast\ctA_{i::})\ast\cX.
\end{align*}
Thus, choosing a step size $\alpha_{other} = \frac{\alpha}{p^2}$ results in the same update as the proposed method.}

We can think of $\frac{1-p}{p^2} \cC \circ (\ctA_{i::}^{\edit{T}}\ast\ctA_{i::})\ast\cX$ as the correction term to ensure that $\m E[g(\cX)] = \nabla F(\cX)$ as required in Lemma~\ref{lem:sgdbound}. In the matrix case under the uniform missing model in~\cite{ma2019stochasticgradientdescentlinear}, the correction term only used the diagonal entries of $\tilde \mA_i^{\edit{T}} \tilde \mA_i$. In our tensor case, we see that the correction term can use many more entries of $\ctA_{i::}^{\edit{T}}\ast\ctA_{i::}$ beyond just a diagonal of a single frontal slice. This flexibility is what allows our algorithm to handle a variety of missing data models. For example, we will find that the column block missing data model requires entries from every frontal slice of the tensor.

Algorithm~\ref{alg:mSGDT} describes our SGD-based algorithm. During each iteration of the descent, a row index $i \in [m] := \{0, 1, \dots, m-1\}$ is chosen uniformly at random, $g(\cX^t)$ is computed on the current iterate based on the choice, and then the iterate $\cX^{\edit{T}}$ is updated via equation~\eqref{eq:SGDupdate}. 

\begin{algorithm}
\caption{mSGDT} \label{alg:mSGDT}
\begin{algorithmic} 
\State \textbf{Input:}  $\cX^0\in \Re^{\ell\times q\times n},$ $\tilde \cA\in \Re^{m\times \ell \times n}$, $\cB\in\Re^{m\times q\times n}$, $p \in (0, 1)$. Closed convex domain $\cW$ that contains the solution $\cX_\star.$
\Procedure{}{$\ctA$, $\cB$, $T$, $p$, $\{\alpha_t\}$}
\State Set $\cC$ satisfying~\eqref{eqn:c1} and~\eqref{eqn:c2} for all $0 \le i < m.$
\For {$t=0,1,\ldots, T - 1 $} 
\State Choose row index $i \in [m]$ uniformly at random
\State 
$g(\cX^{t}) = \frac{1}{p^2} \left(\ctA_{i::}^{\edit{T}}\ast(\ctA_{i::}\ast\cX^t - p \cB_{i::}) \right) - \frac{1-p}{p^2}\cC \circ (\tilde \cA_{i::}^{\edit{T}} \ast \tilde \cA_{i::} ) \ast\cX^t$
   \State $\cX^{t+1} = \proj \left( \cX^t - \alpha_t g(\cX^t) \right) $
\EndFor
\State Output $\cX^{T}$
\EndProcedure
\end{algorithmic}
\end{algorithm}

In order for this algorithm to converge, the choice of $\cC$ is crucial to ensure $\m E[g(\cX)] = \nabla F(\cX)$. Observe that there are two sources of randomness in this expectation of $g$: the choice of the row index $i$ and the binary mask $\cD$. Let $\m E_i[\cdot]$ denote the expectation with respect to the choice of the row index, and let $\m E_{\cD}[\cdot]$ denote the expectation with respect to the mask. Then, the expectation of $g(\cX)$ is the expectation over $i$ and ${\cD}$. That is, 
\[ \m E[g(\cX)] = \m E_i[\m E_{\cD}[g(\cX)]].\]
Now, if $\cC$ can be chosen such that the following equations are true for all $0 \le i < m$,
\begin{align}
    \m E_{\cD} \left[ \cC \circ(\tilde \cA_{i::}^{\edit{T}} \ast \tilde \cA_{i::}) \right] &= p \cC \circ(\cA_{i::}^{\edit{T}} \ast \cA_{i::}), \label{eqn:c1}\\
    \m E_{\cD} \left[ (\mathbb {1} - \cC) \circ(\tilde \cA_{i::}^{\edit{T}} \ast \tilde \cA_{i::}) \right] &= p^2 (\mathbb {1} -  \cC) \circ(\cA_{i::}^{\edit{T}} \ast \cA_{i::})\label{eqn:c2},
\end{align}
where $\mathbb{1}$ is the tensor of all ones, then under Assumption~\ref{assump:model}, it is simple to see that $\m E[g(\cX)] = \nabla F(\cX)$:
\begin{align*}
    \m E[g(\cX)] &= \frac{1}{p^2} \m E \left[\ctA_{i::}^{\edit{T}} \ast (\ctA_{i::}\ast\cX - p \cB_{i::}) - (1-p) \cC \circ (\ctA_{i::}^{\edit{T}}\ast\ctA_{i::})\ast\cX \right]\\ 
    &= \f 1 {p^2} \m E_i \left[ \m E_{\cD} \left[ (\cC + ( \mathbb{1} - \cC) - (1-p) \cC) \circ (\tilde \cA_{i::}^{\edit{T}} \ast \tilde \cA_{i::}) \right] \ast \cX - p\m E_{\cD} \left[ \tilde \cA_{i::}^{\edit{T}} \ast \cB_{i::} \right]\right] \\
    &= \f 1 {p^2} \m E_i \left[ (p\cC + p^2(\mathbb{1} - \cC)  - (1-p) p \cC) \circ (\cA_{i::}^{\edit{T}} \ast \cA_{i::}) \ast \cX - p^2 \cA_{i::}^{\edit{T}} \ast \cB_{i::} \right] \\
    &= \m E_i \left[ \cA_{i::}^{\edit{T}} \ast \cA_{i::} \ast \cX - \cA_{i::}^{\edit{T}} \ast \cB_{i::} \right] \\
    &= \f 1 m \cA^{\edit{T}} \ast (\cA \ast \cX - \cB) \\ 
    &= \nabla F(\cX). \numberthis \label{eqn:Egx}
\end{align*}
Hence, if such a $\cC$ exists, then standard convergence results for SGD should apply. We present two such results: one for a step size proportional to $1/\sqrt {t} $ and the other for a constant step size.
Theorem~\ref{thm:changingstepsize} handles the former case, saying that the objective function will converge to zero as the number of iterations tends to infinity. Theorem~\ref{thm:fixedstepsize} explores the latter case, showing that the error $\|\cX^t - \cX_\ast\|$ in the iterates decreases quickly until a convergence horizon.

Thus, these two theorems imply that the optimal way to use Algorithm~\ref{alg:mSGDT} is to first use a constant step size for quick initial progress before switching to a step size proportional to $1 / \sqrt {t} $. Experimentally, the simplest and best way to switch is to do so after a set number of iterations. See Section~\ref{sec:experiments} for a more in-depth discussion. 

\begin{theorem} \label{thm:changingstepsize} Suppose that Assumptions~\ref{assump:row} and~\ref{assump:model} hold, and suppose there exists a tensor $\cC$ of zeros and ones satisfying~\eqref{eqn:c1} and~\eqref{eqn:c2} for all $0 \le i < m.$

Let $\cX_\star \in \m R^{\ell \times q \times n}$ be such that $\cA \ast \cX_\star = \cB$, let $\cD$ be a binary mask, and let $\tilde \cA = \cD \circ \cA.$ Choosing $\alpha_t = C / \sqrt t$, Algorithm~\ref{alg:mSGDT} converges in expectation such that the error in objective function $F(\cX) = \f 1 {2m} \|\cA \ast \cX - \cB\|^2$ satisfies
\begin{equation*}
\E [ F(\cX^{t}) - F(\cX_\star) ] \leq \left ( \f {K^2} C + CG \right) \f {2 + \log (t) } {\sqrt t},
\end{equation*} 
where $G = \f {4 n^2 R^2} {p^3m} \sum_{i = 0}^{m-1}  \| \cA_{i::} \|^4 +  \f {4 n^{3/2} R} {p^2m} \sum_{i = 0}^{m-1} \| \cA_{i::} \|^3 \| \cB_{i::} \| + \f {2 n}{p^2m} \sum_{i = 0}^{m-1}  \| \cA_{i::} \|^2 \|\cB_{i::} \|^2,$ $R = \max_{\cW_1 \in \cW} \| \cW_1 \|$, and $K$ is a constant satisfying $\sup_{\cW_1 , \cW_2 \in \cal W} \|\cW_1 - \cW_2\| \le K$.
\end{theorem}

\begin{theorem}\label{thm:fixedstepsize}
   Suppose that Assumptions~\ref{assump:row} and~\ref{assump:model} hold, and suppose there exists a Hermitian tensor $\cC$ of zeros and ones satisfying~\eqref{eqn:c1} and~\eqref{eqn:c2} for all $0 \le i < m.$
   
Let $\cX_\star \in \m R^{\ell \times q \times n}$ be such that $\cA \ast \cX_\star = \cB$, let $\cD$ be a binary mask, and let $\tilde \cA = \cD \circ \cA.$
   Furthermore, suppose the objective function $F(\cX) = \f 1 {2m} \| \cA \ast \cX - \cB\|^2$ is $\mu$-strongly convex for some $\mu > 0,$  
   let $L_g = na^2_{max} / p^2$, where $a_{max}$ is the maximum Frobenius norm of a row slice of $\cA$, and let $G_\star = \f {4n^2 R^2}{p^3m} \sum_{i = 0}^{m-1} \| \cA_{i::} \|^4.$ 
   
   Then, with fixed step size $\alpha < 1 / L_g$, Algorithm~\ref{alg:mSGDT} converges with expected error, for all $t > 0,$
   $$\m E[ \| \cX^{t} - \cX_\star\|^2] \le r^t \| \cX^0 - \cX_\star\|^2 + \f {\alpha G_\star}{\mu(1 - \alpha L_g)},$$
   where $r := (1 - 2 \alpha \mu(1 - \alpha L_g)).$
\end{theorem}

Note that these theorems impose some additional conditions on $\cC$ beyond just satisfying~\eqref{eqn:c1} and~\eqref{eqn:c2}. Both require that $\cC$ consists of just zeros and ones -- this is needed in Lemmas~\ref{lem:G},~\ref{lem:G*}, and~\ref{lem:Lg} to compute the constants $G,G_\star,$ and $L_g$. If necessary, one can weaken this condition at the cost of complicating the computation of these constants. Theorem~\ref{thm:fixedstepsize} also imposes that $\cC$ is Hermitian. This is a condition needed to validate a technical step in the proof, see Lemma~\ref{lem:Gver}.

We remark that the combination of~\eqref{eqn:c1} and~\eqref{eqn:c2} says that $\cC$ indicates the entries of $\tilde \cA_{i::}^{\edit{T}} \ast \tilde \cA_{i::}$ whose expected values have a coefficient of $p$, while $\mathbb{1} - \cC$ indicates those entries whose expected values have a coefficient of $p^2$. These conditions are quite unrestrictive, so the results we obtain apply to a broad class of missing data models.

The following corollary directly applies these two theorems to the three missing data models we study. The corollary follows directly from the theorems after we construct $\cC$ for the missing data models. This is done in Proposition~\ref{prop:modelsC}. Note that the portion of the corollary pertaining to the uniform missing model is the same as the results obtained in~\cite{Ma2025Stocha}.
\begin{corollary} \label{cor:application}
Suppose that Assumption~\ref{assump:row} holds, and moreover suppose that the missing data model is either the uniform missing model, the column block missing model, or the frontal slice missing model. Let $\cX_\star \in \m R^{\ell \times q \times n}$ be such that $\cA \ast \cX_\star = \cB$, let $\cD$ be a binary mask under the chosen missing data model, and let $\tilde \cA = \cD \circ \cA.$

Choosing $\alpha_t = C / \sqrt t$, Algorithm~\ref{alg:mSGDT} converges in expectation such that the error in objective function $F(\cX) = \f 1 {2m} \|\cA \ast \cX - \cB\|^2$ satisfies
\begin{equation*}
\E [ F(\cX^{t}) - F(\cX_\star) ] \leq \left ( \f {K^2} C + CG \right) \f {2 + \log (t) } {\sqrt t},
\end{equation*} 
where $G = \f {4 n^2 R^2} {p^3m} \sum_{i = 0}^{m-1}  \| \cA_{i::} \|^4 +  \f {4 n^{3/2} R} {p^2m} \sum_{i = 0}^{m-1} \| \cA_{i::} \|^3 \| \cB_{i::} \| + \f {2 n}{p^2m} \sum_{i = 0}^{m-1}  \| \cA_{i::} \|^2 \|\cB_{i::} \|^2,$ $R = \max_{\cW_1 \in \cW} \| \cW_1 \|$, and $K$ is a constant satisfying $\sup_{\cW_1 , \cW_2 \in \cal W} \|\cW_1 - \cW_2\| \le K$.

Furthermore, suppose $F$ is $\mu$-strongly convex for some $\mu > 0$, let $L_g = na^2_{max} / p^2$, where $a_{max}$ is the maximum Frobenius norm of a row slice of $\cA$, and let $G_\star = \f {4n^2 R^2}{p^3m} \sum_{i = 0}^{m-1} \| \cA_{i::} \|^4.$
   
Then, with fixed step size $\alpha < 1 / L_g$, Algorithm~\ref{alg:mSGDT} converges with expected error, for all $t > 0,$
$$\m E[ \| \cX^{t} - \cX_\star\|^2] \le r^t \| \cX^0 - \cX_\star\|^2 + \f {\alpha G_\star}{\mu(1 - \alpha L_g)},$$
where $r := (1 - 2 \alpha \mu(1 - \alpha L_g)).$
\end{corollary}

\section{Convergence Proofs} \label{sec:convergence}
Besides Lemma~\ref{lem:sgdbound}, the only extra ingredient needed for Theorem~\ref{thm:changingstepsize} is a bound on $\m E[\|g(\cX)\|^2].$ Note that since $\cW$ is bounded, $\m E[\|g(\cX)\|^2]$ is necessarily bounded. But we explicitly calculate $G$ to make the discovered bound more clear. As this is a technical computation, we defer this to the Appendix in Lemma~\ref{lem:G}. Now, the proof of Theorem~\ref{thm:changingstepsize} is straightforward.
\\

\begin{proof} [Proof of Theorem~\ref{thm:changingstepsize}] To use Lemma~\ref{lem:sgdbound}, we need three facts about $F(\cX) = \f 1 {2m} \| \cA \ast \cX - \cB \|^2$ and $g(\cX)$.
\begin{enumerate}
    \item $F(\cal X)$ is convex. This is clear.
    \item $\m E[g(\cX)] = \nabla F(\cX)$. This was shown to follow from~\eqref{eqn:c1} and~\eqref{eqn:c2}, see~\eqref{eqn:Egx}.
    \item $\m E[\|g(\cX)\|^2] \le G,$ where $G = \f {4 n^2 R^2} {p^3m} \sum_{i = 0}^{m-1}  \| \cA_{i::} \|^4 +  \f {4 n^{3/2} R} {p^2m} \sum_{i = 0}^{m-1} \| \cA_{i::} \|^3 \| \cB_{i::} \| + \f {2 n}{p^2m} \sum_{i = 0}^{m-1}  \| \cA_{i::} \|^2 \|\cB_{i::} \|^2,$ $R = \max_{\cW_1 \in \cW} \| \cW_1 \|$, and $K$ is a constant satisfying $\sup_{\cW_1 , \cW_2 \in \cal W} \|\cW_1 - \cW_2\| \le K$. This is Lemma~\ref{lem:G}.
\end{enumerate}
\end{proof}

To prove Theorem~\ref{thm:fixedstepsize}, we will need the following lemma to rewrite $\|g(\cX^t) - g(\cX_\star)\|^2$. Lemma~\ref{lem:Lg} shows that the $L_g$ in the statement of the theorem is a Lipschitz constant for our function $g$, while Lemma~\ref{lem:Gver} shows that there is a smooth $f$ with $\nabla f = g.$ We remark that both in the matrix case in~\cite{ma2019stochasticgradientdescentlinear} and in the tensor case in~\cite{Ma2025Stocha} under the uniform missing data model, it was straightforward to find such an $f$ for their update function. However, the construction for a general $g$ in the form~\eqref{eqn:g} takes more care.
\begin{lemma}[\cite{needell2014stochastic} Lemma A.1] \label{lem:lipinequality}
    If $g$ is a function such that there exists a smooth $f$ with $\nabla f = g$, and $g$ has Lipschitz constant $L$, then $$ \|g(\cX) - g(\cY) \|^2 \le L \langle \cX - \cY, g(\cX) - g(\cY) \rangle.$$
\end{lemma}

\begin{proof} [Proof of Theorem~\ref{thm:fixedstepsize}]

 Let $\cE^{t} = \cX^{t} - \cX_\star$ denote the error at iteration $t$. Let $\E_{t-1}[\cdot]$ denote the expected value conditioned on the previous $t-1$ iterations, i.e. conditioned on $\cX^{t-1}.$ We have 
\begin{align*}
	\E_{t-1} &\edit{[ \| \cE^{t} \|^2 ]} \\&\le \E_{t-1} \left[ \| \cX^{t-1} - \alpha g(\cX^{t-1}) - \cX_\star \|^2 \right] \\
	& = \| \cE^{t-1} \|^2 - 2 \alpha \langle \cE^{t-1}, \E_{t-1} \left[ g(\cX^{t-1}) \right]\rangle  + \alpha^2 \E_{t-1} \left[ \| g(\cX^{t-1}) \|^2 \right]  \\
	& \overset{(a)}{=} \| \cE^{t-1} \|^2 - 2 \alpha \langle \cE^{t-1}, \nabla F(\cX^{t-1}) - \nabla F(\cX_\star)\rangle +  \alpha^2 \E_{t-1} \left[ \| g(\cX^{t-1}) \|^2 \right] \\
& \overset{(b)}{\leq}  \| \cE^{t-1} \|^2 - 2 \alpha \langle \cE^{t-1}, \nabla F(\cX^{t-1}) - \nabla F(\cX_\star)\rangle  
	        + 2 \alpha^2 \E_{t-1} \left [ \| g(\cX^{t-1}) - g(\cX_\star) \|^2 \right] +  2\alpha^2\E_{t-1} \left[ \| g(\cX_\star)\|^2 \right] \\
 &\overset{(c)}{\leq} \| \cE^{t-1} \|^2 - 2 \alpha \langle \cE^{t-1}, \nabla F(\cX^{t-1}) - \nabla F(\cX_\star)\rangle  
	      + 2\alpha^2 L_{g} \langle \cE^{t-1}, \E_{t-1} [ g(\cX^{t-1})] -  \E_{t-1}  [g(\cX_\star)]\rangle   +  2\alpha^2 G_\star \\
	& = \| \cE^{t-1} \|^2 - 2 \alpha (1- \alpha L_g) \langle \cE^{t-1}, \nabla F(\cX^{t-1}) - \nabla F(\cX_\star)\rangle + 2\alpha^2 G_\star \\
		 &\overset{(d)}{\leq} \| \cE^{t-1} \|^2 - 2 \alpha (1- \alpha L_g) \mu \|\cE^{t-1}\|^2  + 2\alpha^2 G_\star \\ 
		& = \left(1 - 2 \alpha \mu (1- \alpha L_g) \right) \| \cE^{t-1} \|^2 + 2\alpha^2 G_\star \\
		& = r \| \cE^{t-1} \|^2 + 2\alpha^2 G_\star .
\end{align*}
Here, (a) holds by recalling $\m E [g(\cX)] = \nabla F(\cX)$ and $\nabla F(\cX_\star) = \f 1 m \cA^{\edit{T}} \ast (\cA \ast \cX_\star - \cB) = 0.$ (b) follows from:
\begin{align*} 
 \m E_{t-1} [ \| g(\cX^{t-1}) \| ^2] &= 
\m E_{t-1} [ \| (g(\cX^{t-1}) - g(\cX_\star)) + g(\cX_\star) \|^2] 
\\ &\le 
\m E_{t-1}[ 2 \| g(\cX^{t-1}) - g(\cX_\star)\|^2 + 2 \|g(\cX_\star)\|^2] \\ 
&\le 2 \m E_{t-1} [ \|g(\cX^{t-1} - g(\cX_\star)\|^2] + 2 \m E_{t-1} [ \|g(\cX_\star)\|^2].
\end{align*} 
(c) uses Lemma~\ref{lem:lipinequality} and replaces $\m E_{t-1} [ \| g(\cX_\star)]\|^2$ with $G_\star$ via Lemma~\ref{lem:G*}. (d) uses the assumption that $F$ is $\mu$-strongly convex.

The desired bound follows by iteratively recursively applying this result, using the Law of Iterated Expectation.
\begin{align*}
    \m E [\| \cE^{t} \|^2] &=
    \m E[ \m E_{t-1} [\| \cE^t \|^2 ]] \\ 
    &\le r\m E[ \| \cE^{t-1} \|^2] + 2 \alpha^2 G_\star \\ 
    & \le r^2 \m E [ \| \cE^{t-2} \|^2] + 2r \alpha^2 G_\star + 2 \alpha^2 G_\star \\ 
     & \enspace \vdots \\ 
    & \le r^t \m E[ \| \cE^0\|^2] + 2\alpha^2 G_\star (r^{t-1} + \dots + r^0) \\ 
    & = r^t \| \cX^0 - \cX_\star\|^2 + \f { 2 \alpha^2 G_\star (1 - r^t)}{1 - r} \\ 
    &= r^t \| \cX^0 - \cX_\star\|^2 + \f { 2 \alpha^2 G_\star (1 - r^t)}{2 \alpha \mu (1 - \alpha L_g)} \\ 
    &= r^t \| \cX^0 - \cX_\star\|^2 + \f {\alpha G_\star (1 - r^t)}{\mu(1 - \alpha L_g)} \\ 
    &\le r^t \| \cX^0 - \cX_\star\|^2 + \f {\alpha G_\star}{\mu(1 - \alpha L_g)}.
\end{align*}
Here, the assumption $\alpha < 1 / L_g$ implies $r = 1 - 2 \alpha \mu(1 - \alpha L_g) < 1.$
\end{proof}

In order to deduce Corollary~\ref{cor:application} from Theorems~\ref{thm:changingstepsize} and~\ref{thm:fixedstepsize}, we need to construct $\cC$ for each of the three missing data models such that the necessary conditions hold.
Proposition~\ref{prop:modelsC} below picks the correct $\cC$.

\begin{proposition} \label{prop:modelsC}
    For the three models, the following choices of $\cC$ are Hermitian tensors of zeros and ones that satisfy~\eqref{eqn:c1} and~\eqref{eqn:c2} for all $0 \le i < m.$
    \begin{enumerate}
        \item (Uniform missing model) $\cC$ is the tensor with ones along the frontal diagonal and zeros elsewhere. Denote this tensor by $\cC_{U}.$
        \item (Column block missing model) $\cC$ is the tensor where every frontal slice is a block diagonal matrix with blocks of size $b$ of ones on the diagonal. Denote this tensor by $\cC_{C}.$
        \item (Frontal slice missing model) $\cC$ is the tensor where the zeroth frontal slice is the matrix of all ones, and all other frontal slices are zero. Denote this tensor by $\cC_{F}.$
    \end{enumerate}
\end{proposition}

\begin{proof}
It is obvious that the entries of each $\cC$ are zeros and ones, and that each $\cC$ is Hermitian.

Thus, we turn to proving~\eqref{eqn:c1} and~\eqref{eqn:c2}. Let $0 \le i < m.$ For $0 \le k < n$, \edit{recall that we view the $k$th frontal slice of $\tilde \cA_{i::}$ as a vector with the notation $\tilde a_k \in \m R^{\ell}.$}  Observe that
\begin{align*}
    \tilde \cA_{i::}^{\edit{T}}\ast\tilde\cA_{i::} = 
    \text{fold} \begin{pmatrix}
        \tilde a_0^{\edit{T}}\tilde a_0 + \tilde a_1^{\edit{T}}\tilde a_1 + \dots + \tilde a_{n-1}^{\edit{T}}\tilde a_{n-1} \\ \tilde a_{n-1}^{\edit{T}} \tilde a_0 + \tilde a_0^{\edit{T}}\tilde a_1 + \dots + \tilde a_{n-2}^{\edit{T}}\tilde a_{n-1} \\ 
        \vdots \\ 
        \tilde a_1^{\edit{T}}\tilde a_0 + \tilde a_2^{\edit{T}}\tilde a_1 + \dots + \tilde a_0^{\edit{T}}\tilde a_{n-1}
    \end{pmatrix}.
\end{align*}
Inside the $\fold{\cdot}$, we have many outer products of vectors, which are easy to analyze. We do casework on the model.

\begin{enumerate}
    \item (Uniform missing model) Under this model, for every pair of coordinates $x, y$, $\m E_{\cD} [(\tilde a^{\edit{T}}_j \tilde a_k)_{xy}]$ is equal to $p^2(a_j^{\edit{T}}a_k)_{xy}$ unless the same two elements of $\cA$ are being multiplied together, in which case there is only a factor of $p$. That is, 
    $$ \m E_{\cD} [(\tilde a^{\edit{T}}_j \tilde a_k)_{xy}] = \begin{cases}
    p^2 (a_j^{\edit{T}} a_k)_{xy}, & j \neq k \text{ or } x \neq y\\ 
    p (a_j^{\edit{T}} a_k)_{xy}, & j = k \text{ and } x = y.
\end{cases}$$
Hence,
\begin{align*}
&\m E_{\cD} [\tilde \cA_{i::}^{\edit{T}} \ast \tilde \cA_{i::}] = p^2 \cA_{i::}^{\edit{T}} \ast \cA_{i::} + (p - p^2)\cC_{U} \circ (\cA_{i::}^{\edit{T}}\ast \cA_{i::}), \\ 
&\m E_{\cD} \left[ \cC_U \circ(\tilde \cA_{i::}^{\edit{T}} \ast \tilde \cA_{i::}) \right] = p \cC_U \circ(\cA_{i::}^{\edit{T}} \ast \cA_{i::}).
\end{align*} 
\item (Column block missing model) In this model, for every pair of coordinates $x, y$, $\m E_{\cD}[(\tilde a_j^{\edit{T}}\tilde a_k)_{xy}]$ is equal to $p^2(a_j^{\edit{T}} a_k)_{xy}$ unless two elements of $\cA$ in the same column block are being multiplied together, in which case there is only a factor of $p$. That is, 
    $$ \m E_{\cD} [(\tilde a^{\edit{T}}_j \tilde a_k)_{xy}] = \begin{cases}
    p^2 (a_j^{\edit{T}} a_k)_{xy}, & x, y \text{ not in the same column block}\\ 
    p (a_j^{\edit{T}} a_k)_{xy}, & x, y \text{ in the same column block.}
\end{cases}$$
Hence,
\begin{align*}
&\m E_{\cD} [\tilde \cA_{i::}^{\edit{T}} \ast \tilde \cA_{i::}] = p^2 \cA_{i::}^{\edit{T}} \ast \cA_{i::} + (p - p^2)\cC_{C} \circ (\cA_{i::}^{\edit{T}}\ast \cA_{i::}), \\ 
&\m E_{\cD} \left[ \cC_C \circ(\tilde \cA_{i::}^{\edit{T}} \ast \tilde \cA_{i::}) \right] = p \cC_C \circ(\cA_{i::}^{\edit{T}} \ast \cA_{i::}).
\end{align*} 
\item (Frontal slice missing model) For every pair of coordinates $x, y$, $\m E_{\cD}[(\tilde a_j^{\edit{T}}\tilde a_k)_{xy}]$ is equal to $p^2(a_j^{\edit{T}} a_k)_{xy}$ unless two elements of $\cA$ in the same frontal slice are being multiplied together, in which case there is only a factor of $p$. That is, 
    $$ \m E_{\cD} [(\tilde a^{\edit{T}}_j \tilde a_k)_{xy}] = \begin{cases}
    p^2 (a_j^{\edit{T}} a_k)_{xy}, & j \neq k\\ 
    p (a_j^{\edit{T}} a_k)_{xy}, & j = k.
\end{cases}$$
Hence,
\begin{align*}
&\m E_{\cD} [\tilde \cA_{i::}^{\edit{T}} \ast \tilde \cA_{i::}] = p^2 \cA_{i::}^{\edit{T}} \ast \cA_{i::} + (p - p^2)\cC_{F} \circ (\cA_{i::}^{\edit{T}}\ast \cA_{i::}), \\ 
&\m E_{\cD} \left[ \cC_F \circ(\tilde \cA_{i::}^{\edit{T}} \ast \tilde \cA_{i::}) \right] = p \cC_F \circ(\cA_{i::}^{\edit{T}} \ast \cA_{i::}).
\end{align*} 
\end{enumerate}
Therefore, regardless of the model,
\begin{align}
&\m E_{\cD} [\tilde \cA_{i::}^{\edit{T}} \ast \tilde \cA_{i::}] = p^2 \cA_{i::}^{\edit{T}} \ast \cA_{i::} + (p - p^2)\cC \circ (\cA_{i::}^{\edit{T}}\ast \cA_{i::}), \label{eqn:c3} \\ 
&\m E_{\cD} \left[ \cC \circ(\tilde \cA_{i::}^{\edit{T}} \ast \tilde \cA_{i::}) \right] = p \cC \circ(\cA_{i::}^{\edit{T}} \ast \cA_{i::}). \label{eqn:c1proof}
\end{align} 
\eqref{eqn:c1proof} is exactly~\eqref{eqn:c1}. Subtracting~\eqref{eqn:c1proof} from~\eqref{eqn:c3} yields the desired~\eqref{eqn:c2}.
\end{proof}

\section{Experiments and Implementation} 
In this section, we present numerical experiments using synthetically generated and real-world data. The code can be found in the following GitHub link: \href{https://github.com/alexxue99/mSGDT}{github.com/alexxue99/mSGDT}
\label{sec:experiments}
\subsection{Synthetic Data}
In this subsection, we simulate mSGDT on synthetic data for each of the three missing data models. We generate elements of $\cA \in \m R^{10^6 \times 20 \times 10}$ and $\cX \in \m R^{20 \times 10 \times 10}$ by drawing i.i.d. from a standard Gaussian distribution. We enforce the assumption that rows of $\cA$ are picked only once by picking rows without replacement and using $10^6$ iterations. 

\edit{After drawing $\cA$ and $\cX$, we run each missing data model for 30 trials, making sure to do a pass over the entire tensor $\cA$ and zero out entries according to the missing data model being tested for each trial. We then plot the mean of the Frobenius norm error, and we include the standard deviation over the 30 trials as a shaded error bar around the mean.}


The constant step size we use is $\alpha = p^2 / 5000$. Note the $p^2$ factor here. A dependence on $p$ is expected since a smaller $p$ means that we have less data and should therefore choose a less aggressive step size. Furthermore, a $p^2$ factor is the natural choice when considering the result $L_g = na^2_{max}/ p ^2$ in Lemma~\ref{lem:Lg} and the requirement of the step size being less than $1 / L_g$ in Theorem~\ref{thm:fixedstepsize}. The 5000 number here was also chosen to satisfy this $1 / L_g$ condition, obtained by using a rough safe estimate for the maximum Frobenius norm of a row slice of a tensor of $\cA$'s dimensions.

For the varying step size, we use $\alpha_t = p^2 /  \sqrt {5000t}$, where the $\sqrt{5000}$ factor was chosen to match the step size with the constant step size at the swap.


\begin{figure}[h]
    \centering
    \includegraphics[width=0.7\textwidth]{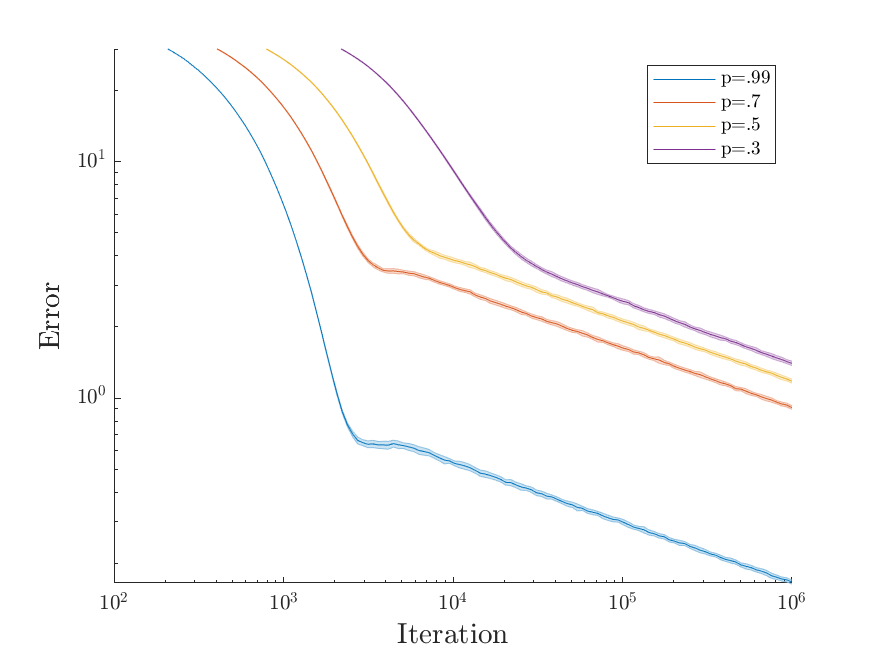}
    \caption{Log-log plot for Algorithm~\ref{alg:mSGDT} under the uniform missing data model, for varying values of $p \in \{.3, .5, .7, .99\}$. The $x$-axis is the iteration, and the $y$-axis is the error, defined as the Frobenius norm of $\cX^t - \cX_\ast$.}
    \label{fig:regular}
\end{figure}

\begin{figure}[h]
    \centering
    \includegraphics[width=0.7\textwidth]{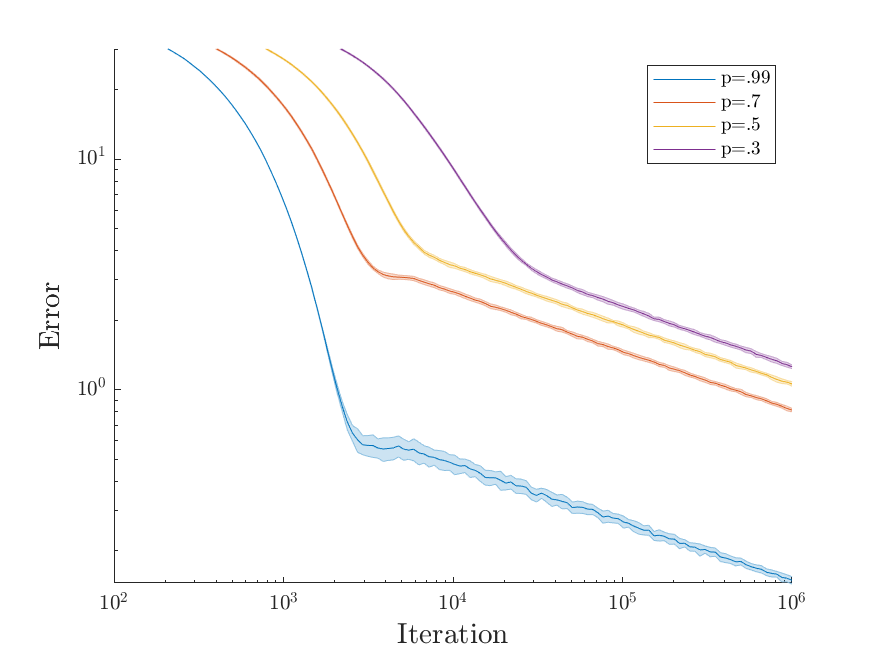}
    \caption{Log-log plot for Algorithm~\ref{alg:mSGDT} under the column block missing data model, for block size $b = 4$ and for varying values of $p \in \{.3, .5, .7, .99\}$. The $x$-axis is the iteration, and the $y$-axis is the error, defined as the Frobenius norm of $\cX^t - \cX_\ast$.}
    \label{fig:column}
\end{figure}

\begin{figure}[h]
    \centering
    \includegraphics[width=0.7\textwidth]{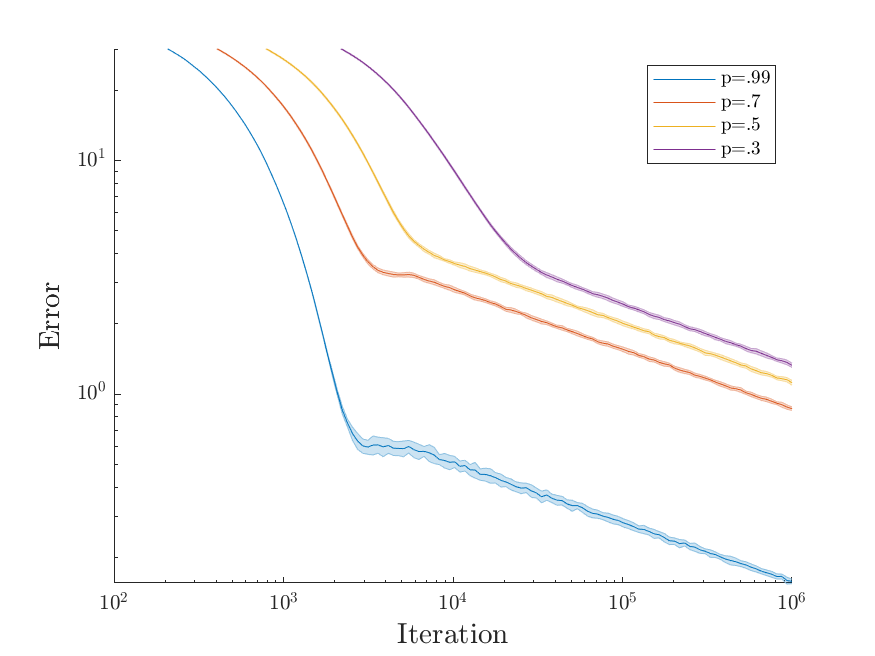}
    \caption{Log-log plot for Algorithm~\ref{alg:mSGDT} under the frontal slice missing data model, for varying values of $p \in \{.3, .5, .7, .99\}$. The $x$-axis is the iteration, and the $y$-axis is the error, defined as the Frobenius norm of $\cX^t - \cX_\ast$.}
    \label{fig:frontal}
\end{figure}

Figures~\ref{fig:regular},~\ref{fig:column}, and~\ref{fig:frontal} show the results for the three missing data models. For all values of $p$, we see an initial fast convergence during the initial iterations with the constant step size. Then, we see steady convergence after swapping to the varying step size, even for a $p$ value as small as 0.3.

\textbf{Remark.} We tested a few other criteria, such as swapping when the norm of the difference between successive iterates drops below a threshold. This criterion attempts to detect the convergence horizon by identifying a plateau in the iterates $\cX$. However, it did not do well in practice. The norm of the update direction
$$g(\cX^{t}) = \frac{1}{p^2} \left(\ctA_{i::}^{\edit{T}}\ast(\ctA_{i::}\ast\cX^t - p \cB_{i::}) \right) - \frac{1-p}{p^2}\cC \circ (\tilde \cA_{i::}^{\edit{T}} \ast \tilde \cA_{i::} ) \ast\cX^t$$
does not tend to zero because of the missing data, even if the current iterate $\cX^t$ is a good estimate of $\cX_\star.$ Thus, it is no surprise that this criterion did not perform well. We also tested a criterion where we estimate $\m E[\ctA_{i::}\ast\cX^t - p \cB_{i::}]$ over a sliding window of iterates, with the idea that its norm should decrease as $\cX^{\edit{T}}$ becomes more accurate. This criterion also failed, likely because the iterates differ greatly between iterations, making $\m E[\ctA_{i::}\ast\cX^t - p \cB_{i::}]$ difficult to estimate correctly over a window.

Thus, we chose to use a simple criterion that swaps after a fixed number of iterations. We chose to use 5000 iterations to better showcase the convergence obtained by the varying step size. In other settings, if it is feasible to estimate $\mu$ and $L_g$, then the convergence ratio $r = (1 - 2 \alpha \mu(1 - \alpha L_g))$ in Theorem~\ref{thm:fixedstepsize} can be estimated, and an appropriate fixed number of iterations can then be chosen.

\FloatBarrier

\subsection{Shuttle Video Data}
Here, we simulate mSGDT on the `shuttle.avi' video data set provided by MATLAB. The video data is stored in the tensor $\cX \in \m R^{288 \times 512 \times 121},$ with each frame being a frontal slice of $\cX.$ As in the previous subsection, we randomly generate Gaussian entries for the tensor $\cA$. We then zero out entries according to the uniform missing data model\edit{, the column block missing model with $b = 8$, or the frontal slice missing model,} with probabilities $p \in \{0.3, 0.7\}.$

Since the tensor $\cX$ is so large in dimension and norm, we initialize the first iterate $\cX^0$ to be filled with the value 128 (as entries in $\cX$ represent grayscale values and are thus contained in [0, 255]) to decrease the maximum possible initial error. Because of the large initial error, we choose a small initial constant step size of $p^2/10^6$ for the first 5000 iterations before swapping to a decreasing step size of $\sqrt{5000} p^2 / 10^6 \sqrt t.$ The $\sqrt{5000}$ factor was chosen to match the step size with the constant step size at the swap. We did not notice a significant improvement in swapping at a later iteration.

Figure~\ref{fig:shuttle} shows the first and last frames of the video data. Figures~\ref{fig:p3} and~\ref{fig:p7} show these same frames from the tensor obtained with mSGDT for \edit{the three missing data models with} different $p \in \{0.3, 0.7\}.$
\edit{In Figure~\ref{fig:p3}, which corresponds to probability $p = 0.3$,} we can see that even with such a low probability, the overall shapes within the video can be distinguished in both the first and last frames \edit{for all three missing data models}, but there are horizontal artifacts that distort the image.

\edit{In Figure~\ref{fig:p7}, which corresponds to the higher probability $p = 0.7,$} the images look much sharper, but there are still some faint artifacts. 
It is worth noting that these artifacts occur frequently in tensor recovery problems, see e.g.~\cite{castillo2025blockgaussseidelmethodstproduct} and~\cite{He2022-jd}.

\begin{figure}[h]
\centering
\begin{minipage}{0.45\linewidth}
  \includegraphics[width=\linewidth]{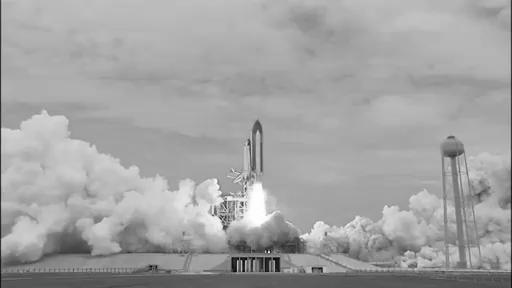}
  \end{minipage}
  \hfill
\begin{minipage}{0.45\linewidth}
\includegraphics[width = \linewidth]{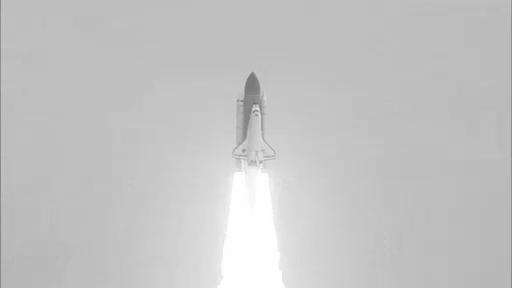}
\end{minipage}
\caption{First and last frames of shuttle dataset.}
\label{fig:shuttle}
\end{figure}

\begin{figure}[h]
    \centering
    \subfloat
    {\includegraphics[width=0.32\linewidth]{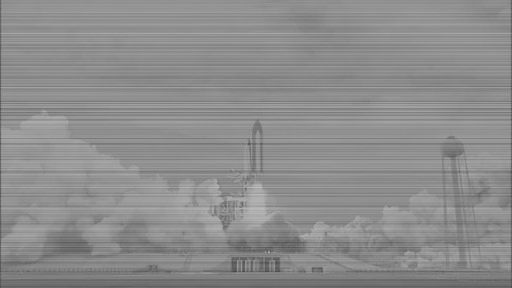}}
    \hfill
    \subfloat
    {\includegraphics[width=0.32\linewidth]{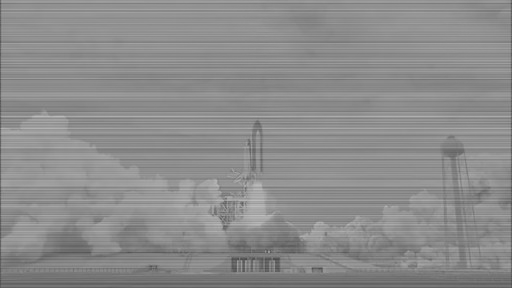}}
    \hfill
    \subfloat
    {\includegraphics[width=0.32\linewidth]{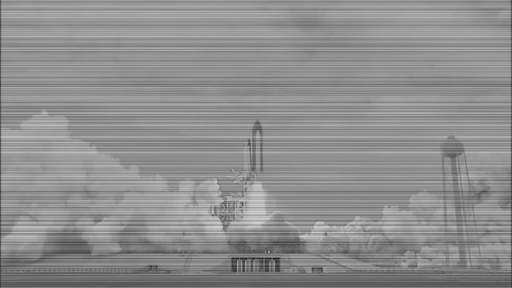}}
    \newline 
    \subfloat
    {\includegraphics[width=0.32\linewidth]{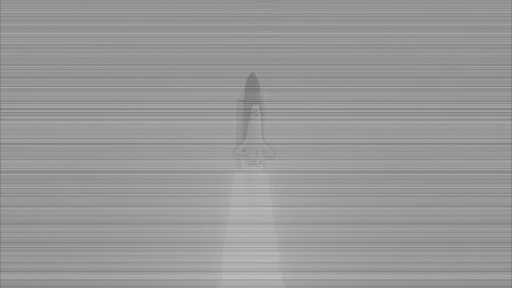}}
    \hfill
    \subfloat
    {\includegraphics[width=0.32\linewidth]{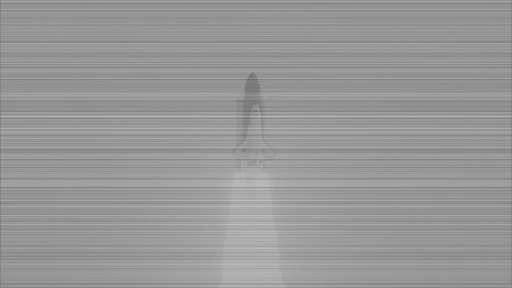}}
    \hfill
    \subfloat
    {\includegraphics[width=0.32\linewidth]{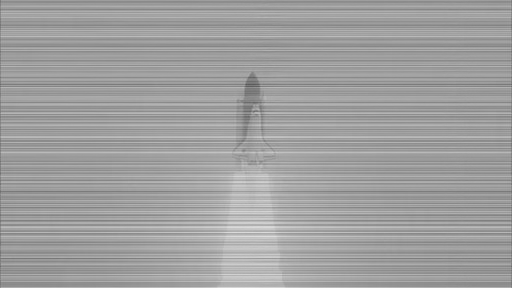}}
    \caption{First (top) and last frames (bottom) of reconstructed tensor. Entries in $\cA$ follow the uniform missing model (left), the column block missing model with $b = 8$ (middle), or the frontal slice missing model (right), each with $p = 0.3$.}
    \label{fig:p3}
\end{figure}

\begin{figure}[h]
    \centering
    \subfloat
    {\includegraphics[width=0.32\linewidth]{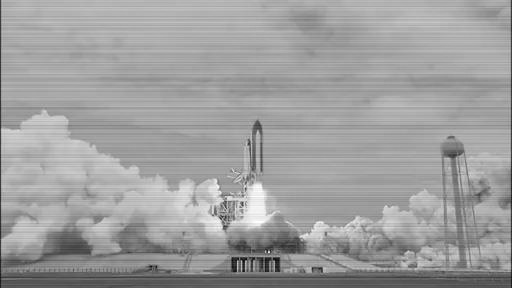}}
    \hfill
    \subfloat
    {\includegraphics[width=0.32\linewidth]{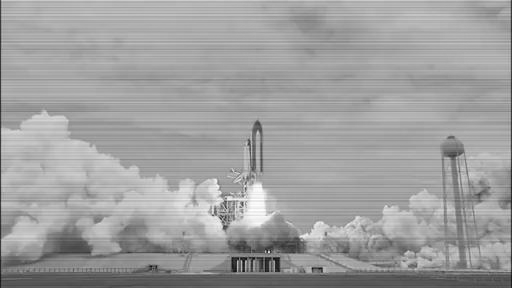}}
    \hfill
    \subfloat
    {\includegraphics[width=0.32\linewidth]{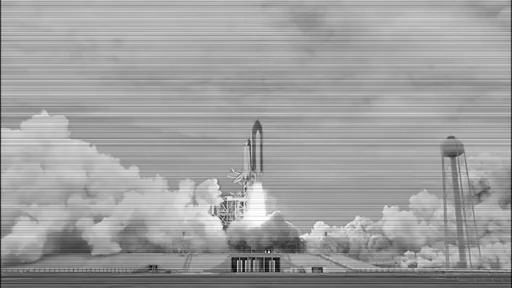}}
    \newline 
    \subfloat
    {\includegraphics[width=0.32\linewidth]{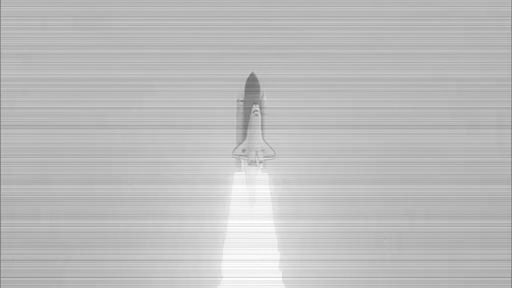}}
    \hfill
    \subfloat
    {\includegraphics[width=0.32\linewidth]{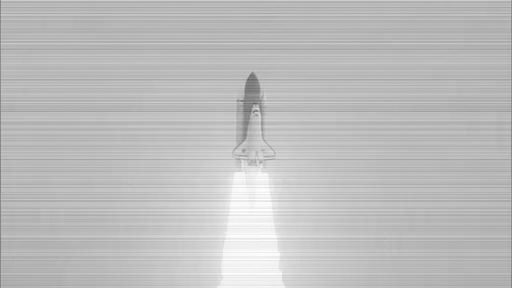}}
    \hfill
    \subfloat
    {\includegraphics[width=0.32\linewidth]{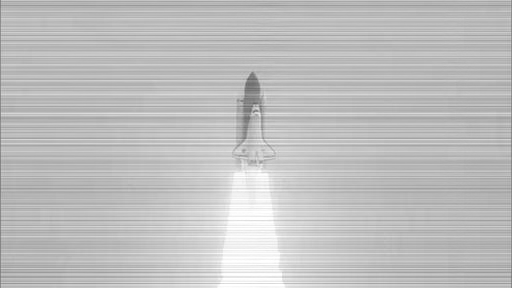}}
    \caption{First (top) and last frames (bottom) of reconstructed tensor. Entries in $\cA$ follow the uniform missing model (left), the column block missing model with $b = 8$ (middle), or the frontal slice missing model (right), each with $p = 0.7$.}
    \label{fig:p7}
\end{figure}

\FloatBarrier

\section{Conclusion}
In this paper, we presented mSGDT, a stochastic iterative method to solve tensor linear systems with missing data. We showed two theoretical convergence results under some general conditions that apply to a large class of missing data models. With a step size proportional to $1/\sqrt t$, mSGDT converges to the correct solution, with faster progress for large values of $p$. With a constant step size, mSGDT makes quick initial progress but eventually converges only to a convergence horizon. Our experiments on synthetic data and video data support these two results.

The key step in proving convergence under a varying step size was to choose an update direction $g$ satisfying~\eqref{eqn:c1} and~\eqref{eqn:c2} for all row indices so that the update direction is unbiased. For the constant step size, we needed to establish various regularity properties of $g$. We hope this framework can be used to study tensor systems with missing data in other settings. Future work could study these systems with products other than the $t$-product, such as the cosine transform product and other tensor-tensor products discussed in~\cite{KERNFELD2015545}. Alternatively, the convergence of iterative methods other than SGD, such as Gauss-Seidel variants, could be studied using a proof technique similar to that of Theorem~\ref{thm:fixedstepsize}.

\section*{Acknowledgements} DN and AX were partially supported by NSF DMS 2408912.

\bibliography{mSGDT}

\section{Appendix}
Lemma~\ref{lem:sub-multiplicative} proves a straightforward sub-multiplicative property of the Frobenius norm. Lemma~\ref{lem:G} gives a bound on $\m E[\|g(\cX)\|^2]$, Lemma~\ref{lem:G*} gives a bound on $\m E[\|g(\cX_\star)\|^2],$ and Lemma~\ref{lem:Lg} gives a Lipschitz constant for $g$. These lemmas closely follow the proof steps done in~\cite{Ma2025Stocha} for the uniform missing data model, adapted to the other models. Lemma~\ref{lem:strongconvex} is a restatement of a lemma appearing in~\cite{Ma2025Stocha} that computes the strongly convex parameter $\mu$ of the objective function $F$. Lastly, Lemma~\ref{lem:Gver} finds an $f$ such that $\nabla f(\cX) = g(\cX).$

\begin{lemma} \label{lem:sub-multiplicative} The Frobenius norm $\|\cX\| = \sqrt{\sum_{i, j, k}  \cX_{i, j,k} ^2}$ is $\sqrt n$-sub-multiplicative. That is, 
$$\|\cX\ast \cY\| \le \sqrt n \|\cX\|\|\cY\|.$$
Moreover, $\sqrt n$ is the best possible factor.
\end{lemma}

\begin{proof}  We have 
$$ \unfold{\cX \ast \cY} =  \begin{pmatrix}
        \mX_0 & \mX_{n-1} & \mX_{n-2} & \ldots & \mX_{1}\\
        \mX_{1} & \mX_{0} & \mX_{n-1} & \ldots & \mX_{2}\\
        \vdots & \vdots & \vdots & \ddots & \vdots \\
        \mX_{n-1} & \mX_{n-2} & \mX_{n-3} & \ldots & \mX_{0}\\
        \end{pmatrix} \begin{pmatrix}
            \mY_0 \\ \mY_1 \\ \vdots \\ \mY_{n-1}
        \end{pmatrix}.
        $$
    Thus, since the Frobenius norm is sub-multiplicative for matrices,
    \begin{align*}
        \| \cX \ast \cY \| &\le \sqrt n \| \cX \| \| \cY\|,
    \end{align*} 
    as desired.

For showing that $\sqrt n$ is the best possible factor, consider $\cX \in \m R^{m \times \ell \times n}$ and $\cY \in \m R^{\ell \times q \times n}$ with every entry of both tensors equal to 1. Then every entry of $\unfold{\cX \ast \cY} \in \m R^{m \times q \times n}$ is $\ell n$, so $$ \| \cX \ast \cY \| = \ell n\sqrt {m q n},$$ while $$\|\cX \| \| \cY\| = \sqrt{m \ell n} \sqrt{ \ell q n} = \ell n\sqrt{ m q}.$$

\end{proof}

\begin{lemma}[edited from~\cite{Ma2025Stocha}] \label{lem:G} Suppose Assumption~\ref{assump:model} holds. Let $g$ be defined as in~\eqref{eqn:g}, and suppose that $\cC$ is a tensor of zeros and ones satisfying~\eqref{eqn:c1} and~\eqref{eqn:c2} for all $0 \le i < m$. Let $R$ be an upper bound on $\|\cX\|$. Then
    $\m E[\|g(\cX)\|^2] \le G,$ where 
    $$G =\f {4 n^2 R^2} {p^3m} \sum_{i = 0}^{m-1}  \| \cA_{i::} \|^4 +  \f {4 n^{3/2} R} {p^2m} \sum_{i = 0}^{m-1} \| \cA_{i::} \|^3 \| \cB_{i::} \| + \f {2 n}{p^2m} \sum_{i = 0}^{m-1}  \| \cA_{i::} \|^2 \|\cB_{i::} \|^2.$$
    
\end{lemma}
\begin{proof}
We have
\begin{align*}
\E [\| &g(\cX) \|^2]  =  \E \bigg[ \bigg \| \frac{1}{p^2} \left(\tilde \cA_{i::}^{\edit{T}}\ast(\tilde \cA_{i::}\ast\cX - p \cB_{i::}) \right) - \frac{(1-p)}{p^2}\cC \circ(\tilde \cA_{i::}^{\edit{T}}\ast\tilde \cA_{i::})\ast \cX \bigg \| ^2 \bigg] \\ 
&{\leq} \frac{2}{p^4} \underbrace{\E \left[\norm{ \tilde \cA_{i::}^{\edit{T}}\ast(\tilde \cA_{i::}\ast\cX - p \cB_{i::}) }^2 \right]}_\text{(A)} + \frac{2(1-p)^2}{p^4} \underbrace{\E \left[\| \cC \circ(\tilde \cA_{i::}^{\edit{T}}\ast\tilde \cA_{i::})\ast\cX \|^2\right]}_\text{(B)}. \numberthis \label{eqn:g^2inequality}
\end{align*} 
Let's bound (A) first. 
\begin{align*} 
  \E \left[\norm{ \tilde \cA_{i::}^{\edit{T}}\ast(\tilde \cA_{i::}\ast\cX - p \cB_{i::}) }^2 \right] & \leq n \E \left[\| \tilde \cA_{i::} \|^2 \|\tilde \cA_{i::}\ast\cX - p \cB_{i::}\|^2\right] \\ 
& \leq n\E \left[\norm{ \cA_{i::} }^2 \|\tilde \cA_{i::}\ast\cX - p \cB_{i::}\|^2\right]\\
&= n \E_i \big [\norm{ \cA_{i::} }^2 \E_{\cD} [\|\tilde \cA_{i::}\ast\cX - p \cB_{i::}\|^2] \big ], \numberthis \label{eqn:A}
\end{align*}
where the first inequality holds by Lemma~\ref{lem:sub-multiplicative}, the second by the trivial bound $\|\tilde \cA_{i::}\|^2 \le \| \cA_{i::}\|^2$, and the last by recalling $\m E[\cdot] = \m E_i[ \m E_{\cD}[\cdot]]$ and noting that $\cA_{i::}$ is independent of the binary mask.

Note
\begin{align*}
    \E_{\cD} [\|\tilde \cA_{i::}\ast\cX - p \cB_{i::}\|^2] &= 
    \E_{\cD} [\|\tilde \cA_{i::}\ast\cX\|^2] - 2p \m E_{\cD} \langle \tilde \cA_{i::} \ast \cX, \cB_{i::} \rangle + p^2 \|\cB_{i::}\|^2.
\end{align*}
For the first term, we compute
\begin{align*}
    \E_{\cD} [\|\tilde \cA_{i::}\ast\cX\|^2] &= 
    \m E_{\cD} \langle \tilde \cA_{i::} \ast \cX, \tilde \cA_{i::} \ast \cX \rangle \\
    &\overset{(a)}{=} \m E_{\cD} \langle \cX^{\edit{T}} \ast \tilde \cA_{i::}^{\edit{T}} \ast \tilde \cA_{i::} \ast \cX, \cI \rangle \\ 
    &= \langle \cX^{\edit{T}} \ast \m E_{\cD} [\tilde \cA_{i::}^{\edit{T}} \ast \tilde \cA_{i::}] \ast \cX, \cI \rangle \\ 
    &\overset{(b)}{=} p^2 \langle \cX^{\edit{T}} \ast \cA_{i::}^{\edit{T}} \ast \cA_{i::} \ast \cX, \cI \rangle + (p - p^2) \langle \cX^{\edit{T}} \ast (\cC \circ (\cA_{i::}^{\edit{T}} \ast \cA_{i::})) \ast \cX, \cI \rangle \\ 
    &= p^2 \langle \cA_{i::} \ast \cX, \cA_{i::} \ast \cX \rangle + (p - p^2) \langle \cC \circ ( \cA_{i::}^{\edit{T}} \ast \cA_{i::}) \ast \cX, \cX \rangle \\
    &\overset{(c)}{\le} p^2 \| \cA_{i::} \ast \cX \|^2 + (p - p^2)  \| \cC \circ (\cA_{i::}^{\edit{T}} \ast \cA_{i::}) \ast \cX  \| R \\
    &\overset{(d)}{\le} p^2 \| \cA_{i::} \ast \cX \|^2 + (p - p^2) n  \|\cA_{i::} \|^2 R^2
\end{align*}
Here, (a) moves the $\tilde \cA_{i::} \ast \cX$ on the right to the left by using a transpose, leaving the identity tensor $\cI$. (b) applies~\eqref{eqn:c1} and~\eqref{eqn:c2}. (c) uses Cauchy-Schwarz, and (d) uses Lemma~\ref{lem:sub-multiplicative} and the assumption that $\cC$ consists of zeros and ones. Thus,
\begin{align*}
    \E_{\cD} [\|\tilde \cA_{i::}\ast\cX - p \cB_{i::}\|^2] &\le p^2 \| \cA_{i::} \ast \cX \|^2 + (p - p^2) n  \|\cA_{i::} \|^2 R^2
     - 2p \m E_{\cD} \langle \tilde \cA_{i::} \ast \cX, \cB_{i::} \rangle + p^2 \|\cB_{i::}\|^2 \\ 
     &= p^2 \| \cA_{i::} \ast \cX \|^2 + (p - p^2) n  \| \cA_{i::} \|^2 R^2 - 2p^2 \langle \cA_{i::} \ast \cX, \cB_{i::} \rangle + p^2 \|\cB_{i::}\|^2. \numberthis \label{eqn:lem5help}
\end{align*}
Therefore, for (A) we obtain the bound
\begin{align*} 
\E &\left[\norm{ \tilde  \cA_{i::}^{\edit{T}}\ast(\tilde \cA_{i::}\ast\cX - p \cB_{i::}) }^2 \right]
\\&
\le n \m E_{i} \left [ \|\cA_{i::}\|^2 \left( p^2 \| \cA_{i::} \ast \cX \|^2 + (p - p^2) n  \| \cA_{i::} \|^2 R^2 - 2p^2 \langle \cA_{i::} \ast \cX, \cB_{i::} \rangle + p^2 \|\cB_{i::}\|^2\right) \right]
\\
&\le n \m E_i \left[ p^2  n R^2 \| \cA_{i::} \|^4 + (p - p^2)n R^2 \|\cA_{i::} \|^4 + 2p^2 \sqrt n R\|\cA_{i::} \|^3 \| \cB_{i::} \| + p^2 \| \cA_{i::} \|^2 \| \cB_{i::} \|^2 \right] \\
&= \sum_{i = 0}^{m-1} \f {p n^2 R^2} {m} \| \cA_{i::} \|^4 + \sum_{i = 0}^{m-1} \f {2 p^2  n^{3/2} R} m \| \cA_{i::} \|^3 \| \cB_{i::} \| + \sum_{i = 0}^{m-1} \f {p^2 n}{m} \| \cA_{i::} \|^2 \|\cB_{i::} \|^2.
\end{align*}
The second inequality applies Cauchy-Schwarz once and Lemma~\ref{lem:sub-multiplicative} multiple times.

For (B), 
\begin{align*} 
\E \left[\| \cC \circ(\tilde \cA_{i::}^{\edit{T}}\ast\tilde \cA_{i::})\ast\cX \|^2\right] & \overset{(a)}{\le} n R^2 \m E \left [ \| \cC \circ (\tilde \cA_{i::}^{\edit{T}} \ast \tilde \cA_{i::} ) \| ^2 \right] \\ 
& \overset{(b)}{\le} n R^2 \m E \left [ \|\tilde \cA_{i::}^{\edit{T}} \ast \tilde \cA_{i::} \| ^2 \right] \\ 
& \overset{(c)}{=} n^2 R^2 \m E \left [ \|\tilde \cA_{i::}^{\edit{T}}\|^2 \| \tilde \cA_{i::} \|^2 \right] \\ 
& = n^2 R^2 \m E \left [ \| \tilde \cA_{i::} \|^4 \right] \\
& \overset{(d)}{\le} p n^2 R^2 \m E_i \left[ \| \cA_{i::} \|^4 \right] \\ 
& \le \f {p n^2 R^2} m \sum_{i = 0}^{m-1} \| \cA_{i::} \|^4 . \numberthis \label{eqn:B}
\end{align*} 
Here, (a) and (c) apply Lemma~\ref{lem:sub-multiplicative}, and (b) uses the assumption that entries of $\cC$ are zeros and ones. (d) is the bound 
$$ \m E_{\cD} \left[ \| \tilde \cA_{i::} \| ^4 \right] = \m E_{\cD} \left( \sum_{j, k} \tilde \cA_{ijk}^2 \right)^2 \le p \left ( \sum_{j, k} \cA_{ijk}^2 \right)^2 = p \| \cA_{i::}\|^4.$$

Thus,
\begin{align*}
&\E \left[ \| g(\cX) \|^2 \right]\\
&\le \sum_{i = 0}^{m-1} \f {2 n^2 R^2} {p^3m} \| \cA_{i::} \|^4 + \sum_{i = 0}^{m-1} \f {4 n^{3/2} R} {p^2m} \| \cA_{i::} \|^3 \| \cB_{i::} \| + \sum_{i = 0}^{m-1} \f {2 n}{p^2m} \| \cA_{i::} \|^2 \|\cB_{i::} \|^2 + \f {2 n^2 R^2(1 - p)^2} {p^3m} \sum_{i = 0}^{m-1} \| \cA_{i::} \|^4 \\
& \le \f {4 n^2 R^2} {p^3m} \sum_{i = 0}^{m-1}  \| \cA_{i::} \|^4 +  \f {4 n^{3/2} R} {p^2m} \sum_{i = 0}^{m-1} \| \cA_{i::} \|^3 \| \cB_{i::} \| + \f {2 n}{p^2m} \sum_{i = 0}^{m-1}  \| \cA_{i::} \|^2 \|\cB_{i::} \|^2.
\end{align*} 
\end{proof}

\begin{lemma}\label{lem:G*} Suppose Assumption~\ref{assump:model} holds. Let $g$ be defined as in~\eqref{eqn:g}, and suppose that $\cC$ is a tensor of zeros and ones satisfying~\eqref{eqn:c1} and~\eqref{eqn:c2} for all $0 \le i < m$. Let $R$ be an upper bound on $\|\cX_\star \|$. Then
    $\m E[\|g(\cX_\star)\|^2] \le G_\star,$ where $G_\star = \f {4n^2 R^2}{p^3 m} \sum_{i = 0}^{m-1} \|\cA_{i::}\|^4.$
\end{lemma}
\begin{proof}

In Lemma~\ref{lem:G} we had bound (A) by $n \E_i \big [\norm{ \cA_{i::} }^2 \E_{\cD} [\|\tilde \cA_{i::}\ast\cX - p \cB_{i::}\|^2] \big ]$ in~\eqref{eqn:A}, and we later bound the $\m E_{\cD}$ term by $$\E_{\cD} [\|\tilde \cA_{i::}\ast\cX - p \cB_{i::}\|^2] \le p^2 \| \cA_{i::} \ast \cX \|^2 + (p - p^2) n  \| \cA_{i::} \|^2 R^2 - 2p^2 \langle \cA_{i::} \ast \cX, \cB_{i::} \rangle + p^2 \|\cB_{i::}\|^2$$ in~\eqref{eqn:lem5help}. Note that we can rewrite this as 
$$ p^2 \| \cA_{i::} \ast \cX \|^2 + (p - p^2) n  \| \cA_{i::} \|^2 R^2 - 2p^2 \langle \cA_{i::} \ast \cX, \cB_{i::} \rangle + p^2 \|\cB_{i::}\|^2 = p^2 \| \cA_{i::} \ast \cX - \cB_{i::} \|^2 + (p - p^2) n \| \cA_{i::} \|^2 R^2.$$

Substituting in $\cX = \cX_\star$, we get $\|\cA_{i::} \ast \cX - \cB_{i::}\| = 0$, so (A) is bounded by 
$$ n \m E_i \left[ \| \cA_{i::}\|^2 (p -p^2) n \|\cA_{i::}\|^2 R^2 \right] \le n^2 R^2 (p - p^2) \m E_i \left[ \| \cA_{i::} \|^4 \| \right].$$
Combining this with~\eqref{eqn:B} in~\eqref{eqn:g^2inequality}, we obtain
\begin{align*}
    \m E[\|g(\cX_\star)\|^2] &\le 
    \f {2n^2 R^2 (1 - p)}{p^3} \m E_i \left[\| \cA_{i::}\|^4 \right] + \f {2n^2 R^2 (1-p)^2}{p^3} \m E_i [\|\cA_{i::}\|^4] \\
    & \le \f {4n^2 R^2 }{p^3} \m E_i [ \| \cA_{i::} \|^4 ] \\ 
    &= \f {4n^2 R^2}{p^3 m} \sum_{i = 0}^{m-1} \|\cA_{i::} \|^4.
\end{align*}
\end{proof}

\begin{lemma}\label{lem:Lg}
Let $g$ be defined as in~\eqref{eqn:g}, and suppose that $\cC$ is a tensor of zeros and ones.
Then, $g$ is Lipschitz continuous, with the supremum of the Lipschitz constant over all possible binary masks bounded above by $L_g := na_{max}^2 / p^2$, where $a_{max}$ is the maximum Frobenius norm of a row slice of $\cA.$
\end{lemma}
\begin{proof}
We have 
\begin{align*}
\|g(\cX) - g(\cY)\| &= \bigg \| \left(\frac{1}{p^2} \ctA_{i::}^{\edit{T}}\ast\ctA_{i::} - \frac{1-p}{p^2}\cC \circ (\ctA_{i::}^{\edit{T}}\ast\ctA_{i::})\right) \ast (\cX - \cY) \bigg \| \\ 
 &\le \f {\sqrt n}{p^2} \left \| \ctA_{i::}^{\edit{T}}\ast\ctA_{i::} - (1-p)\cC \circ (\ctA_{i::}^{\edit{T}}\ast\ctA_{i::}) \right \| \left \| \cX - \cY \right \| \\
&\le \f {\sqrt{n}}{p^2} \left \| \tilde \cA_{i::}^{\edit{T}} \ast\tilde \cA_{i::} \right \| \left \| \cX - \cY \right \| \\
&\le \f n {p^2} \| \tilde \cA_{i::}\|^2 \| \cX - \cY\| \\
&\le \f n {p^2} \| \cA_{i::} \|^2 \| \cX - \cY \| \\
&\le \f {n a_{max}^2}{p^2} \|\cX - \cY\|.
\end{align*}
The first and third inequalities apply Lemma~\ref{lem:sub-multiplicative}, while the second inequality uses the assumption that the entries of $\cC$ are zeros and ones.
\end{proof}

The next lemma describes the strongly convex parameter of the objective function in terms of $\bdiag{\widehat \cA},$ where $\widehat \cA$ is the tensor obtained by applying the discrete Fourier transform to the tubes of $\cA,$ and $\bdiag{\cdot}$ returns the block diagonal matrix formed by the frontal slices of the input tensor. A \textit{tube} of a tensor is obtained by fixing the first two indices and varying the third index. We view tubes as (column) vectors in $\m R^n$. We denote the $(i,j)$th tube of $\cA$ as $\cA_{ij:}$.
\begin{lemma}[\cite{Ma2025Stocha} Lemma 7]\label{lem:strongconvex} Recall our objective function $F(\cX) = \f 1 {2m} \| \cA \ast \cX - \cB\|^2.$ Let $\cal A = \cal U  \ast\cal S \ast \cal V^{\edit{T}}$ be the T-SVD of $\cal A$ \cite{KILMER2011641}. Assume that $\cA$ is tall, i.e. $m > \ell.$ Let $\sigma_{min}$ be the smallest singular value of $\bdiag{\widehat \cA}$. Then $F$ is $\sigma_{min}^2/m$-strongly convex.
\end{lemma}

\begin{lemma}
Let $g$ be defined as in~\eqref{eqn:g}, and suppose that $\cC$ is Hermitian. Then, there exists a smooth $f$ such that $\nabla f = g.$
    \label{lem:Gver}
\end{lemma}

\begin{proof}


    Recall that 
    \begin{equation} 
    \nabla( \f 1 2 \|\cQ\ast\cX - \cR\|^2) = \cQ^{\edit{T}}\ast(\cQ\ast\cX - \cR) \label{eqn:grad}
    \end{equation}
    for any tensors $\cQ, \cR$ independent of $\cX$. Hence 
    \begin{equation*}
    \nabla \left(\f 1 {2p^2} \|\tilde \cA_{i::}\ast\cX - p \cB_{i::} \|^2\right) = 
    \f 1 {p^2} \tilde \cA_{i::}^{\edit{T}}\ast(\tilde \cA_{i::} \ast\cX - p \cB_{i::}),
    \end{equation*}
    which is the first term in $g$. Therefore, it suffices to show that there exists a function $f$ with 
    \begin{equation}\nabla f(\cX) = \cC \circ (\tilde \cA_{i::}^{\edit{T}} \ast \tilde \cA_{i::}) \ast \cX.\label{eqn:lemfgrad}
    \end{equation}
    We will define $f$ as the sum of terms of the form $\f 1 2 \| \cQ \ast \cX\|^2$ and use~\eqref{eqn:grad} to compute the gradient of $f$.
    
     To that end, let $\cH = \cC \circ (\tilde \cA_{i::}^{\edit{T}} \ast \tilde \cA_{i::}).$ We will decompose $\cH$ and define tensors based on the decomposition to construct such an $f$. Because $\cC$ and $\tilde \cA_{i::}^{\edit{T}} \ast \tilde \cA_{i::}$ are Hermitian, $\cH$ is also Hermitian, so we only decompose the first half of the frontal slices. For each integer $1 \le k < n / 2$, write the $k$th frontal slice
    $$ \mH_k = \sum_{j = 0}^{N_k} u_{jk} v_{jk}^{\edit{T}}$$ as sum of outer products of vectors in $\m R^{\ell \times 1}$. For the special case $k =  n /2$ (which occurs if $n$ is even), we will instead define $u_{jk}$ and $v_{jk}$ such that 
    $$ \mH_{n/2} = \sum_{j = 0}^{N_{n/2}} u_{jk} v_{jk}^{\edit{T}} + v_{jk} u_{jk}^{\edit{T}}.$$
    Note that such a decomposition must exist because $\mH_k$ is a Hermitian matrix.
    
    
    For each $0 \le j \le  N_k$, we will construct a tensor $\cT^{(j,k)} \in \m R^{1 \times \ell \times n}$ corresponding to $u_{jk}$ and $v_{jk}$ by defining its frontal slices. Since the frontal slices are elements of $\m R^{1 \times \ell}$, we think of them as (row) vectors, and we denote them with the lowercase vector notation $t_{k'}^{(j, k)}$, where $0 \le k' < n.$ We define
    $$ t_{k'}^{(j, k)} = \begin{cases}
        u_{jk}^{\edit{T}}, & k' = 0 \\ 
        v_{jk}^{\edit{T}}, & k' = k \\ 
        0, & \text{otherwise}
    \end{cases}$$
    Then 
    \begin{align*}
     \cT^{(j, k)\edit{T}}\ast\cT^{(j,k)} = 
    \text{fold} \begin{pmatrix}
        t^{(j,k)\edit{T}}_0 t^{(j,k)}_0 + t^{(j,k)\edit{T}}_1 t^{(j,k)}_1 + \dots + t^{(j,k)\edit{T}}_{n-1} t^{(j,k)}_{n-1} \\ t^{(j,k)\edit{T}}_{n-1} t^{(j,k)}_0 + t^{(j,k)\edit{T}}_0 t^{(j,k)}_1 + \dots + t^{(j,k)\edit{T}}_{n-2} t^{(j,k)}_{n-1} \\ 
        \vdots \\ 
        t^{(j,k)\edit{T}}_1 t^{(j,k)}_0 + t^{(j,k)\edit{T}}_2 t^{(j, k)}_1 + \dots + t^{(j,k)\edit{T}}_0 t^{(j,k)}_{n-1}
    \end{pmatrix}
\end{align*}
satisfies, for each $0 \le k' < n$, 

$$
  (\cT^{(j, k)\edit{T}}\ast\cT^{(j,k)})_{k'} = \begin{cases} 
    u_{jk} v_{jk}^{\edit{T}}, & k' = k, k \neq n/2\\ 
    (u_{jk} v_{jk}^{\edit{T}})^{\edit{T}}, & k' = n-k, k \neq n/2 \\ 
    u_{jk} v_{jk}^{\edit{T}} + v_{jk} u_{jk}^{\edit{T}}, & k' = k, k = n/2 \\ 
    u_{jk} u_{jk}^{\edit{T}} + v_{jk} v_{jk}^{\edit{T}}, & k' = 0\\ 
    0, & \text{otherwise.}
    \end{cases}
$$
Therefore, summing across all $0 \le j \le N_k,$
\begin{align*}
 \left ( \sum_{j = 0}^{N_k} \cT^{(j, k)\edit{T}}\ast\cT^{(j,k)}\right)_{k'} &= \begin{cases} 
    \mH_k = \sum_{j = 0}^{N_k} u_{jk} v_{jk}^{\edit{T}} , & k' = k, k \neq n/2 \\[5pt]
   \mH_k^{\edit{T}} =  \sum_{j = 0}^{N_k} (u_{jk} v_{jk}^{\edit{T}})^{\edit{T}} , & k' = n-k, k \neq n/2 \\[5pt]
    \mH_{n/2} = \sum_{j = 0}^{N_k} u_{jk} v_{jk}^{\edit{T}} + v_{jk} u_{jk}^{\edit{T}} , & k' = k, k = n/2 \\[5pt]
    \sum_{j = 0}^{N_k} u_{jk} u_{jk}^{\edit{T}} + v_{jk} v_{jk}^{\edit{T}}, & k' = 0\\ 
    0, & \text{otherwise.}
    \end{cases} \
\end{align*} 
Now, because $\cH$ is Hermitian, the tensor $\sum_{k = 1}^{\lfloor n / 2 \rfloor} \sum_{j = 0}^{N_k} \cT^{(j, k)\edit{T}}\ast\cT^{(j,k)}$ has frontal slices exactly matching those of $\cH$, except with zeroth frontal slice equal to $\sum_{k = 1}^{\lfloor n /2 \rfloor} \sum_{j = 0}^{N_k} u_{jk} u_{jk}^{\edit{T}} + v_{jk} v_{jk}^{\edit{T}}$. That is,
$$ \sum_{k = 1}^{\lfloor n /2 \rfloor} \sum_{j = 0}^{N_k} \cT^{(j, k)\edit{T}}\ast\cT^{(j,k)} - \sum_{k = 1}^{\lfloor n /2 \rfloor} \sum_{j = 0}^{N_k} (u_{jk}^{\edit{T}})^{\edit{T}} \ast u_{jk}^{\edit{T}} + (v_{jk}^{\edit{T}})^{\edit{T}} \ast  v_{jk}^{\edit{T}} = \cH - \cH_{::0},$$
where $u_{jk}^{\edit{T}}$ (resp. $v_{jk}^{\edit{T}}$) here is understood to be the $1 \times \ell \times n$ tensor with zeroth frontal slice equal to the vector $u_{jk}^{\edit{T}}$ (resp. $v_{jk}^{\edit{T}})$ and all other frontal slices equal to zero.

Thus, by~\eqref{eqn:grad}, 
$$\nabla \left [ \sum_{k = 1}^{\lfloor n / 2 \rfloor} \sum_{j = 0}^{N_k} \f 1 2 \| \cT^{(j, k)} \ast \cX \|^2  - \sum_{k = 1}^{\lfloor n /2 \rfloor} \sum_{j = 0}^{N_k} \left( \f 1 2 \| u_{jk}^{\edit{T}} \ast \cX \|^2 + \f 1 2 \| v_{jk}^{\edit{T}} \ast \cX \|^2 \right) \right] = \cH \ast \cX - \cH_{::0} \ast \cX.$$

To show~\eqref{eqn:lemfgrad}, it now suffices to construct an $f$ such that 
\begin{equation} \nabla f(\cX) = \cH_{::0} \ast \cX.\label{eqn:lemfgrad2}\end{equation}

Observe that there exist vectors $\{u_{j}\}_{0 \le j \le N_0}$ with choices of signs $\alpha_j \in \{\pm 1\}$ such that $\mH_0$ can be written as $\mH_0 = \sum_{j = 0}^{N_0} \alpha_j u_{j} u_{j}^{\edit{T}}$. Indeed, $\mH_0$ is Hermitian, so each off-diagonal entry $(\mH_0)_{xy}$ can be obtained by choosing a $u_{j}$ with nonzero entries at indices $x$ and $y$ and zeros elsewhere. Then, the diagonal entries can be obtained by choosing a $u_j$ with one nonzero entry and by choosing an appropriate sign $\alpha_j$.

Finally, note that the choice
$$ f(\cX) := \sum_{j = 0}^{N_0} \f {\alpha_j} 2 \| u_{j}^{\edit{T}} \ast \cX\|^2$$ 
satisfies~\eqref{eqn:lemfgrad2}, where $u_j^{\edit{T}}$ is understood to be the $1 \times \ell \times n$ tensor with zeroth frontal slice equal to the vector $u_j^{\edit{T}}$ and all other frontal slices equal to zero.
\end{proof} 
\end{document}